\documentclass[openany, a4paper, 11pt]{article}
\usepackage{amsmath,amssymb,amsthm}
\usepackage[latin1]{inputenc}
\usepackage{version,tabularx,multicol}
\usepackage{graphicx,float,psfrag}
\usepackage{stmaryrd}
 \usepackage{color}

\theoremstyle{plain}
\newtheorem{theorem}{Theorem}[section]

 \newtheorem{corollary}[theorem]{Corollary}

\newtheorem{lem}[theorem]{Lemma}
 \newtheorem{lemma}[theorem]{Lemma}

 \newtheorem{remark}[theorem]{Remark}

 \def\beqlb{\begin{eqnarray}}\def\eeqlb{\end{eqnarray}}
 \def\beqnn{\begin{eqnarray*}}\def\eeqnn{\end{eqnarray*}}

 \def\ar{\!\!&}

 \def\mbb{\mathbb}

 \def\qed{\hfill$\Box$\medskip}

\newcommand{\bcen}{\begin{center}}
\newcommand{\ecen}{\end{center}}
\newcommand{\bgeqn}{\begin{equation}}
\newcommand{\edeqn}{\end{equation}}

\def\az{\alpha}
\def\bz{\beta}
\def\gz{\gamma}

\def\dz{\delta}

\def\ez{\epsilon}

\def\rar{\rightarrow}

\def\l{\left}
\def\r{\right}

 \def\ar{\!\!\!&}

\begin{document}


\begin{center}{\Large\textbf{ On large deviation rates for sums associated with Galton-Watson processes }\footnote{Supported by the Fundamental Research Funds for the Central Universities (2013YB59) and NSFC
  (No. 11201030, 11371061).}}\end{center}

\date{\today}

\bigskip

\centerline{ Hui He\footnote{ \textit{E-mail address:} {
hehui@bnu.edu.cn }} }

\medskip

\centerline{Laboratory of Mathematics and Complex Systems, }

\centerline{ School of Mathematical Sciences, Beijing Normal
University,}

\smallskip

\centerline{ Beijing 100875, People's Republic of China}

\bigskip


\bigskip

{\narrower{\narrower{\narrower

\begin{abstract}
Given a super-critical Galton-Watson process $\{Z_n\}$ and a positive sequence $\{\epsilon_n\}$, we study the limiting behaviors of $P(S_{Z_n}/Z_n\geq\epsilon_n)$ 
with sums $S_{n}$ of i.i.d. random variables $X_i$ and $m=E[Z_1]$. We assume that we are in Schr\"oder case with $EZ_1\log Z_1<\infty$ and $X_1$ is in the domain of attraction of an $\alpha$-stable law with $0<\alpha<2$. As a by-product,  when $Z_1$ is sub-exponentially distributed, we further obtain the convergence rate of $\frac{Z_{n+1}}{Z_n}$ to $m$ as $n\rightarrow\infty$.
\end{abstract}

\smallskip

\noindent\textit{Key words and phrases.} Galton-Watson process, domain of attraction, stable distribution, slowly varying function, large deviation, Lotka-Nagaev estimator,
Schr\"oder constant.
\smallskip

\noindent\textit{AMS 2010 subject classifications.} {60J80, 60F10}

\smallskip

\noindent\textbf{Abbreviated Title:} LDP for sums

\par}\par}\par}

\bigskip\bigskip

\section{Introduction and Main Results}
\subsection{Motivation}
Let $Z=(Z_n)_{n\geq1}$ be a super-critical Galton-Watson process with $Z_0=1$ and offspring distribution $\{p_k: k\geq0\}$. Define $m=\sum_{k\geq1}kp_k>1$. We assume in this paper that $p_0=0$ and $0<p_1<1$.

It is known that $Z_{n+1}/Z_n\overset{a.s.}{\rar}m$ and $Z_{n+1}/Z_n$ is the so-called Lotka-Nagaev estimator of $m$; see Nagaev \cite{Na69}.  This estimator has been used in  studying amplification rate and the initial number of molecules for amplification process in a quantitative polymerase chain reaction experiment; see \cite{JP96,JP98} and \cite{Pi04}. Concerning the Bahadur efficiency of the estimator leads to investigating the large deviation behaviors of $Z_{n+1}/Z_n$. In fact, it was proved in \cite{Na69} that if $\sigma^2=Var(Z_1)\in (0, \infty)$, then
\beqlb\label{nage}
\lim_{n\rar\infty}P\l(m^{n/2}\l(\frac{Z_{n+1}}{Z_n}-m\r)<x\r)=\int_0^{\infty}
\Phi\l(\frac{x\sqrt{u}}{\sigma}\r)\omega(u)du,
\eeqlb
where $\Phi$ is the standard normal distribution function and $\omega$ denotes the continuous density function of $W:\overset{a.s.}{=}\lim_{n\rar\infty}Z_n/m^n$. In \cite{A94}, Athreya showed that if $p_1m^r>1$ and $E[Z_1^{2r+\dz}]<\infty$ for some $r\geq1$ and $\dz>0$, then
 $$
\lim_{n\rar\infty}\frac{1}{p_1^n}P\l(\l|\frac{Z_{n+1}}{Z_n}-m\r|\geq\ez\r)\text{ exits finitely;}
$$
see also \cite{AV97}. Later, Ney and Vidyashankar \cite{NV03} weakened the assumption and were able to obtain the rate of convergence of  Lotka-Nagaev estimator by studying the asymptotic properties of harmonic moments of $Z_n$, where it was assumed that $P(Z_1\geq x)\sim ax^{1-\eta}$ for some $\eta>2$ and $a>0$. See \cite{NV04} for some further results.

Recently, Fleischmann and Wachtel \cite{FW08} considered a generalization of above problem by studying sums indexed by $Z$; see also \cite{NV04}. More precisely, let $X=(X_n)_{n\geq1}$ denote a family of i.i.d. real-valued random variables. They investigated the large deviation probabilities for ${S_{Z_n}}/{Z_n}$: the convergence rate of
$$
P\l(\frac{S_{Z_n}}{Z_n}\geq\ez_n\r),
$$
as $n\rar\infty$, where $\ez_n\rar0$ is a positive sequence and
$$S_n:={X_1+X_2+\cdots+X_n}.$$ In fact, if $X_1\overset{d}{=}Z_1-m$, then
$$
\frac{S_{Z_n}}{Z_n}\overset{d}{=}\frac{Z_{n+1}}{Z_n}-m.
$$
The assumption in \cite{FW08} is that $E[Z_1\log Z_1]<\infty$, $E[X_1^2]<\infty$ and $P(X_1\geq x)\sim a x^{-\eta}$ for some $\eta>2$, which implies that $X_1$ is in the domain of attraction of normal distributions.

Motivated by above mentioned works, the main purpose of this paper is trying to study the convergence rates of $Z_{n+1}/Z_n$ under weaker conditions. We shall use the framework of \cite{FW08} but we  assume that $E[Z_1\log Z_1]<\infty$ and $X_1$ is in the domain of attraction of a stable law; see Assumptions A and B below. Then we answer a question in \cite{FW08}; see (a) in Remark 11 there. In particular, we further obtain the convergence rate of $Z_{n+1}/Z_n$ under the assumption $P(Z_1>x)\sim {L(x)}{x^{-\bz}}$ for some $1<\bz<2$ and some slowly varying function $L$, which partially improves Theorem 3 in \cite{NV03}.


For proofs, we shall use the strategy of \cite{FW08}. However, our arguments are deeply involved because of the lack of high moments and the perturbations of slowly varying functions. We overcome those difficulties by using Fuk-Nagaev's inequalities, estimation of growth of random walks, large deviation probabilities for sums under
sub-exponentiallity and establishing the asymptotic properties of
\beqlb\label{NegaI}
E[Z_n^{-t}L(\ez_n Z_n)],\quad t>0, \quad \text{as }n\rar\infty.
\eeqlb

 In the next section, Section \ref{subsecBA}, we will give our basic assumptions on $Z$ and $X$. Our main results will be presented in Section \ref{subsecMR}. We prove
 Fuk-Nagaev's inequalities and establish the asymptotic properties of (\ref{NegaI}) in Section \ref{SecPre}. The proofs of main results will be given in Section \ref{SecPro}. With $C, c$, etc., we denote
positive constants which might change from line to line.


\subsection{Basic Assumptions}\label{subsecBA}

Define $F(x)= P(X_1\leq x).$ We  make the following assumption:

\noindent {\bf Assumption A:} 
\begin{itemize}

\item $P(X_1\geq x)\sim x^{-\beta}L(x)$, where $\beta>0$ and $L$ is a slowly varying function;

\item If $\ez_n\rar0$, we assume that $L$ is bounded away from $0$ and $\infty$ on every compact subset of $[0, \infty)$.

\item $X_1$ is in the domain of attraction of an $\az$-stable law with $0<\az<2$; 

\item $E[X_1]=0$ if $1<\az<2$;

\item $E[Z_1\log Z_1]<\infty;$

\item$ p_0=0, 0<p_1<1.$
\end{itemize}
From the Assumption, it is easy to see that $\az\leq \bz.$ The last term in the Assumption means that we are in the Schr\"oder case. In fact, we only need to assume $0<p_0+p_1<1.$
\begin{remark} The second term in the Assumption is technical. In fact, by  Theorem 1.5.6 in \cite{BGT89} for any $\eta>0$ and $a>0$,  there
exist two positive constants $C_{\eta}$ such that, for any
$y>a,\, z>a$,
 \bgeqn
 \label{regular}
\frac{L(z)}{L(y)}\leq
C_{\eta}\max\l(\left(\dfrac{z}{y}\right)^{\eta},\left(\dfrac{z}{y}\right)^{-\eta}\r).
 \edeqn
 And if $L$ is bounded away from $0$ and $\infty$ on every compact subset of $[0, \infty)$, then (\ref{regular}) holds for any $y>0, z>0$.
 \end{remark}
\begin{remark}
\label{remark1.1} Under Assumption A we have that there exists a function $b(k)$ of regular variation of index
$1/\az$ such that
\beqlb\label{bstable}
b(k)^{-1}S_k\overset{d}{\rar}U_{s},
\eeqlb
where $U_{s}$ is an $\az$-stable random variable; see \cite{[F71]} and \cite{ST94}. Without loss of generality, we may and will assume
that function $b$ is continuous and monotonically increasing from
$\mbb R^{+}$ onto $\mbb R^{+}$ and $b(0)=0$; see
\cite{[F71]}. We also have that
$$
b(x)=x^{1/\az}s(x),\quad x>0,
$$
where $s:(0,\infty)\rightarrow(0,\infty)$ is a slowly varying
function. Then (\ref{regular}) also holds for $s$ with $y\geq1, z\geq1$.
\end{remark}

Define \beqlb\label{mu12}\mu(1;x)=\int_{-x}^xyF(dy),\quad\mu(2;x)=\int_{-x}^xy^2F(dy).\eeqlb
 Under Assumption A,  by arguments in \cite{[F71]}, we have  as $x\rar+\infty$,
\beqlb\label{ppp}
\frac{1-F(x)}{1-F(x)+F(-x)}\rar p_+,\quad \frac{F(-x)}{1-F(x)+F(-x)}\rar p_-,\quad  p_++p_-=1
\eeqlb
and
\beqlb\label{muasy}
\frac{x^2[1-F(x)+F(-x)]}{\mu(2;x)}\rar \frac{2-\az}{\az},\qquad\mu(2;x)\sim \begin{cases}\frac{\az}{2-\az}x^{2-\az}R(x),& \text{if } p_+=0;\\ \frac{\bz p_+}{2-\bz}x^{2-\bz}L(x), &\text{if } 0<p_+<1;\\ \frac{\bz}{2-\bz}x^{2-\bz}L(x), &\text{if } p_+=1,\end{cases}
\eeqlb
where $R$ is a slowly varying function.
Furthermore, the function $b$ in (\ref{bstable}) must satisfy:
as $x\rar+\infty$,
\beqlb\label{LB}
x[1-F(b(x))]\rar C{p_+}\frac{2-\az}{\az},\quad
xF(-b(x))\rar C{p_-}\frac{2-\az}{\az};
\eeqlb
see (5.25) in \cite{[F71]}. In particular, it is implied in above that
if $p_+=0$, then $F(-x)\sim x^{-\az}R(x)$ as $x\rar+\infty.$
Then for some technically  reasons, we also need to make the following assumptions.

\bigskip

\noindent {\bf Assumption B:}
\begin{itemize}
\item $U_{s}$ is  strictly stable;

\item If $1<\az<2$, we assume that $\liminf_{x\rar+\infty}s(x)\in (0, +\infty] $;

\item If $0<p_+<1$ and $\az=1$, we assume that $\mu(1;x)=0$ for all $x>0$;

\item If $p_+=0$, we assume $\az<\bz$;

\item If $1<\az<2$ and $p_+>0$, we assume
$$
\limsup_{n\rar+\infty}\frac{F(-b(n)/[\log n]^{1/\az})}{(\log n)F(-b(n))}\leq 1.
$$
\end{itemize}
\begin{remark}  The assumption that $U_s$ is strictly stable implies that, when $\az=1$, we must have $\az=\bz$ and {\it the skewness parameter}  of $U_{s}$  is $0$.
The 2nd term in Assumption B will be used to deduce (\ref{liminfS2}) which is required in Lemma \ref{Na4aa}. The 3rd term is used in {\it Step 2} in Lemma \ref{Na4} to find a good upper bound for $P(x)$, which appears in Theorem 1.2 in \cite{Na79}. The last two terms are required in Theorems 9.2 and 9.3 in \cite{DDS08}, which are needed in our proofs.
\end{remark}
\begin{center}
{\it From now on,   Assumptions A and B are in force.}
\end{center}

\subsection{Main Results}\label{subsecMR}
Before presenting the main results, we first introduce some notation. Recall $b(x)$ from (\ref{bstable}). Define $J(x)=xb(x)^{-1}$ and
$$l(x)=\inf\{y\in [0,\infty): J(y)>x\}.$$
According to Theorem 1.5.12 in \cite{BGT89}, $l(x)$ is an asymptotic inverse  of $J$; i.e.; $$ \quad l(J(x))\sim J(l(x))\sim x, \text{ as } x\rar+\infty.$$ Define $l(\ez_n^{-1})=l_n$.
Note that $l$ is also regular varying function with index $\frac{\az-1}{\az}.$  Denote by $f(s)$ the generating function of our offspring law.  Define $\gamma$ (Schr\"oder constant) by
$$
f'(0)=m^{-\gamma}=p_1.
$$ For $1<\az<2$ and $\az<\bz$, let
 $$
\chi_n:=\frac{l_n^{\gz-\bz}m^{(\bz-1-\gz)n} b(l_n)^{\bz}}{L(l_n^{-1} b(l_n)m^n)}=\frac{b(l_n)^{\gz}}{(\ez_nm^n)^{\gz-\bz}L(\ez_nm^n)m^n}.
$$ For $0\leq t<\gz+1$, define
\beqlb\label{Ibz}
I_{t}=
\int_0^{\infty}u^{1-t}\omega(u)du.
\eeqlb
\begin{remark}
As $u\rar0+$, there exist constants $0<C_1<C_2<\infty$ such that
\beqlb\label{omega}
C_1<\frac{\omega(u)}{u^{\gamma-1}}<C_2.
\eeqlb
See \cite{D71}, \cite{CLR14} and references therein for related results.  So the assumption $E[Z_1\log Z_1]<\infty$, together with (\ref{omega}), implies that $I_t$ is finite; see Theorem 8.12.7 in \cite{BGT89}.
\end{remark}
We are ready to present our main results.
As illustrated in \cite{NV03}, there is a ``phase transition'' in rates depending on $\gz$. Thus we will have three different cases in regard to $\gz$ and $\bz$. We  first consider the case of $\gz>\bz-1$.
\begin{theorem}\label{main2}  Let $0<\az<1$. Assume that  $\ez_n m^n b(m^n)^{-1}\rar +\infty$ and $\ez_n\rar+\infty$ as $n\rar\infty$.
 If $\gamma>\beta-1$,
then
\beqlb\label{main2a}
\lim_{n\rar\infty}m^{(\bz-1)n}\ez_n^{\beta}L(\ez_nm^n)^{-1}P( S_{Z_n}/Z_n\geq \ez_n )=I_{\bz}.
\eeqlb


\end{theorem}
\begin{theorem}\label{main3} Let  $1\leq\az<2$. Assume that $\ez_nm^nb(m^n)^{-1}\rar +\infty$ as $n\rar\infty$ and $\gz>\bz-1$. 
 \begin{enumerate}
\item[(i)]   Assume $1<\az<2,$ $p_+=0$ and $\ez_n\rar0$. If\, $\lim_{n\rar\infty}\chi_n=0 $,
then (\ref{main2a}) holds.

\item[(ii)]  Assume $1<\az<2,$ $p_+=0$ and $\ez_n\rar0$. If\, $\lim_{n\rar\infty}\chi_n= \infty $, then
   \beqlb\label{main32c}
  V_I\ar\leq\ar\varliminf_{n\rar\infty}l_n^{-\gz}m^{\gz n}P( S_{Z_n}/Z_n\geq  \ez_n )\cr\ar\leq\ar \varlimsup_{n\rar\infty}l_n^{-\gz}m^{\gz n}P( S_{Z_n}/Z_n\geq  \ez_n )\leq V_S,
 \eeqlb
  where
  \beqnn
  V_I\ar=\ar\varliminf_{u\downarrow0}u^{1-\gz}\omega(u)\int_0^{\infty}u^{\gz-1}P(U_{s}\geq u^{\frac{\az-1}{\az}})du,\cr
  V_S\ar=\ar \varlimsup_{u\downarrow0}u^{1-\gz}\omega(u)\int_0^{\infty}u^{\gz-1}P(U_{s}\geq u^{\frac{\az-1}{\az}})du.
  \eeqnn

 \item[(iii)]  Assume $1<\az<2,$ $p_+=0$ and $\ez_n\rar0$. If\, $\lim_{n\rar\infty}\chi_n= y\in(0,\infty)$, then
  \beqnn
  V_I+yI_{\bz}
  \ar\leq\ar\varliminf_{n\rar\infty}l_n^{-\gz}m^{\gz n}P( S_{Z_n}/Z_n\geq  \ez_n )\cr\ar\leq\ar \varlimsup_{n\rar\infty}l_n^{-\gz}m^{\gz n}P( S_{Z_n}/Z_n\geq  \ez_n )\leq V_S+yI_{\bz},
 \eeqnn

\item[(iv)] Assume  $p_+>0$ and $\ez_n\rar\ez\in (0, \infty)$.
Then (\ref{main2a}) holds.
 \end{enumerate}
\end{theorem}
\begin{remark}
The assumption $p_+=0$ implies that $U_{s}$ is a spectrally negative $\az$-stable random variable with mean $0$ and skewness parameter $-1$. By (1.2.11) in \cite{ST94}, we have  $$\int_0^{\infty}u^{\gz-1}P(U_{s}\geq u^{\frac{\az-1}{\az}})du<\infty.$$
\end{remark}

As an application of (iv) in above theorem by taking $\ez_n=\ez$, we  immediately get the following result, which improves the corresponding result in Theorem 3 in \cite{NV03}, where it is assumed that $L$ is a constant function.

\begin{corollary}  If $P(Z_1>x)\sim
x^{-\beta}L(x)$ for  $1<\beta<2$ and $\gz>\bz-1$, then
\beqlb\label{A941}\lim_{n\rar\infty} m^{(\bz-1)n}L(m^n)^{-1}P\l(\frac{Z_{n+1}}{Z_n}-m\geq\ez\r)=I_{\bz}\ez^{-\beta}.
\eeqlb
\end{corollary}
\begin{remark}\label{A94RE}
In fact, by (\ref{Na46}) below, one may prove that
$$\lim_{n\rar\infty} m^{(\bz-1)n}L(m^n)^{-1}P\l(m-\frac{Z_{n+1}}{Z_n}\geq\ez\r)=0.$$
\begin{proof}
(iv) in Theorem \ref{main3} implies (\ref{A941}).
\end{proof}

Next, we consider the case of $\gz=\bz-1$. Let $d$ be the greatest common divisor of the set $\{j-i: i\neq j , p_jp_i>0\}.$
\begin{theorem}\label{Cmain2} Suppose  $0<\az<1$ and $\bz>1$. Assume that  $\ez_n m^n b(m^n)^{-1}\rar +\infty$ and $\ez_n\rar+\infty$ as $n\rar\infty$.
 If $\gamma=\beta-1$,
then
\beqlb\label{Cmain2a}
d\liminf_{u\downarrow0}u^{1-\gz}\omega(u)\ar\leq \ar \liminf_{n\rar\infty}\frac{\ez_n^{\bz}P(S_{Z_n}/Z_n\geq\ez_n)}{ \sum_{1\leq k\leq m^n}\frac{L(\ez_n k)}{k m^{\gz n}}}\cr\ar\leq\ar\limsup_{n\rar\infty}\frac{\ez_n^{\bz}P(S_{Z_n}/Z_n\geq\ez_n)}{ \sum_{1\leq k\leq m^n}\frac{L(\ez_n k)}{k m^{\gz n}}}\leq d\limsup_{u\downarrow0}u^{1-\gz}\omega(u).
\eeqlb
\end{theorem}

Define $$ \pi_n=\frac{l_n^{\gz}\ez_n^{\bz}}{\sum_{1\leq k\leq m^n}\frac{L(\ez_n k)}{k } }.$$

\begin{theorem}\label{Cmain3} Let   $1<\az<2$. Assume that $\ez_n\rar0$, $\ez_nm^nb(m^n)^{-1}\rar +\infty$.
 \begin{enumerate}
\item[(i)]   Assume $p_+=0$ and $\gz=\bz-1.$ If $\pi_n\rar0 $,
then (\ref{Cmain2a}) holds.

\item[(ii)]  Assume $p_+=0$ and $\gz=\bz-1.$ If $\pi_n\rar +\infty $, then
  (\ref{main32c}) holds.

 \item[(iii)]  Assume $p_+=0$ and $\gz=\bz-1.$ If $\pi_n\rar y\in(0,\infty)$, then
  \beqnn
  V_I+yd\liminf_{u\downarrow0}u^{1-\gz}\omega(u)
  \ar\leq\ar\liminf_{n\rar\infty}l_n^{-\gz}m^{\gz n}P( S_{Z_n}/Z_n\geq \ez_n )\cr\ar\leq\ar \limsup_{n\rar\infty}l_n^{-\gz}m^{\gz n}P( S_{Z_n}/Z_n\geq \ez_n )\leq V_S+yd\limsup_{u\downarrow0}u^{1-\gz}\omega(u),
 \eeqnn

\item[(iv)] Assume  $p_+>0$ and $\gz=\bz-1$.
Then (\ref{Cmain2a}) holds with $\ez_n$ replaced by any $\ez>0$.
 \end{enumerate}
\end{theorem}
\begin{remark} If $L$ is a constant function, then (\ref{Cmain2a}) can be replaced by
$$
\lim_{n\rar\infty}n^{-1}\ez_n^{\bz}m^{\gz n}P(S_{Z_n}/Z_n\geq\ez_n)=\frac{1}{\Gamma(\bz-1)}\int_1^{m}Q(E[e^{-v W}])v^{\bz-2}dv,
$$
where  \beqlb\label{Q}Q(s)=\sum_{k=1}q_ks^k=\lim_{n\rar\infty}\frac{f_n(s)}{m^{-\gz n}},\quad 0\leq s<1,\quad q_k=\lim_{n\rar\infty}{P(Z_n=k)}{m^{\gz n}}\eeqlb and $f_n$ denotes the iterates of $f$. See  Proposition 2 in \cite{A94} for $Q(s)$ and $(q_k)_{k\geq1}$. The key is the limiting behavior of $E[Z_n^{-\gz}L(\ez_n Z_n)]$ as $n\rar\infty$; see  Theorem 1 in \cite{NV03} and  Remark \ref{remNN} below in this paper.
\end{remark}

Finally, we consider the case of $\gz<\bz-1$.
\begin{theorem}\label{SCmain3}   If  $1<\az<2$ and $\gz<\bz-1$ or $E[X_1^{1+\gz}1_{\{X_1>0\}}]<\infty$, then for any $\ez>0$,
$$
\lim_{n\rar\infty}m^{\gz n}P(S_{Z_n}/Z_n\geq\ez)= \sum_{k\geq1}q_kP(S_k\geq\ez k).
$$
\end{theorem}

\begin{corollary}\label{SCmain3Co}
 If $P(Z_1>x)\sim x^{-\bz}L(x)$ for $1<\bz<2$ and $\gz<\bz-1$ or $E[Z_1^{1+\gz}]<\infty$, then
 $$
\lim_{n\rar\infty}m^{\gz n}P\l(\l|\frac{Z_{n+1}}{Z_n}-m\r|\geq\ez\r) =\sum_{k\geq1}q_k\phi(k,\ez),
$$
where $\phi(k,\ez)=P(|\frac{1}{k}\sum_{i=1}^k\xi_i-m|>\ez)$ and $(\xi_i)_{i\geq1}$ are i.i.d. random variables with the same distribution as $Z_1.$
\end{corollary}
\begin{remark}
When $L$ is a constant function and $P(Z_1>x)\sim x^{-\bz}L$, the above result has been proved in \cite{NV03}. Theorem 1 and  Corollary 1 in \cite{A94} also proved the same result under the assumption $E[Z_1^{2a+\dz}]<\infty$ and $p_1m^a>1$ for some $a\geq1$ and $\dz>0$.
\end{remark}

We also  generalize (\ref{nage}) to the stable setting.
\begin{theorem}\label{mainx}
Assume that $0<\az<2$. If $\ez_n m^nb(m^n)^{-1}\rar x\in(-\infty, +\infty)$, then
\beqlb\label{mainx1}
\lim_{n\rar\infty}P( S_{Z_n}/Z_n\geq \ez_n )=\int_0^{\infty}P\l(U_s\geq u^{\frac{\az-1}{\az}}x\r) \omega(u)du.
\eeqlb
\end{theorem}

As an application of above theorem, the following result generalizes (\ref{nage}); see Theorem 3 in \cite{Na69}.
\begin{corollary} Assume that $1<\bz<2$ and $P(Z_1>x)\sim x^{-\bz}L(x)$ as $x\rar+\infty$.
Then for every $ x\in(-\infty, +\infty)$,
\beqlb\label{corox}
\lim_{n\rar\infty}P\left(\frac{m^n}{b(m^n)}\left(\frac{Z_{n+1}}{Z_n}-m\right)\leq x\right)
= \int_0^{\infty}P\l(U_{s}\leq u^{\frac{\bz-1}{\bz}}x\r) \omega(u)du.
\eeqlb
\end{corollary}
\begin{proof}
Obviously, $Z_1-m$ is in the domain of attraction of $\bz$-stable law. Using Theorem \ref{mainx} with $\ez_n=xb(m^n)m^{-n}$ gives  (\ref{corox}).
\end{proof}

\begin{remark}It is possible to generalize some results above to the setting that $(X_i)_{i\geq1}$ are not independent; see \cite{MW13, MW14} and references therein for related results.
\end{remark}
\section{Preliminaries}\label{SecPre}

\subsection{Fuk-Nagaev inequalities}
The following result is parallel to Lemma 14 in \cite{FW08} where $X_1$ has finite variance.
\begin{lemma} For any $0<\az<1$,  $r>0$ and $k\geq1$, 
\beqlb\label{Na1a}
P(S_k\geq \ez_nk)\leq \begin{cases}kP(X_1\geq r^{-1}\ez_nk)+c_r\ez_n^{-\bz r}k^{(1-\bz)r},&\beta<1\\
k P(X_1\geq r^{-1}\ez_nk)+ c_r\ez_n^{-tr }k^{(1-t)r},& \beta\geq1 \end{cases}.
\eeqlb
hold for $t\in(\az, 1]\cap(\az, \bz).$
\end{lemma}
\begin{proof}
 By Theorem 1.1 in \cite{Na79}, we have 
for any $0<t\leq1$,
\beqnn
P(S_k\geq x)\leq kP(X_1\geq y)+\exp\left\{\frac{x}{y}-\frac{x}{y}\ln\left(\frac{xy^{t-1}}{kA(t;0,y)}+1\right)\right\}
\eeqnn
with  $
A(t;0,y)=E[X_1^t\cdot1_{\{0\leq X_1\leq y\}}],$ which gives
\beqlb\label{Na3b}
P(S_k\geq \ez_nk)\leq kP(X_1\geq r^{-1}\ez_nk)+\l(\frac{eE[X_1^t; 1_{\{0\leq X_1\leq r^{-1}\ez_n k\}}]}{r^{1-t}\ez_n^t k^{t-1}}\r)^r.
\eeqlb
Noting that as $x\rar+\infty$, $P(X_1\geq x)\sim x^{-\beta}L(x)$, we have for $ x>1$,
\beqlb\label{tndm}E[X_1^{t}; 1_{\{0\leq X_1\leq x\}}]\leq \begin{cases}Cx^{t-\bz},&\bz<t; \\ C_t,&t<\beta.\end{cases}\eeqlb
And if $x\leq 1$, obviously we have
\beqlb\label{tndma}E[X_1^{t}; 1_{\{0\leq X_1\leq x\}}]\leq C(1\vee x^{t-\bz}).\eeqlb
Then if $\beta<1$, applying 
(\ref{Na3b}) with  $\bz<t$, together with (\ref{tndm}) and (\ref{tndma}), yields (\ref{Na1a}). 
If $\beta\geq1$, with the help of (\ref{tndm}) and (\ref{tndma}), taking any $\az<t\leq1$ and $r>0$ also implies (\ref{Na1a}). 
\end{proof}

%

\subsection{Harmonic moments}
It is well-known that
$$
W_n:=m^{-n}Z_n\overset{a.s.}{\rar}  W;
$$
see \cite{FW07}.
We further have the global limit theorem:
\beqlb\label{global}
\lim_{n\rar\infty}P(Z_n\geq x m^n)=\int_x^{\infty}\omega(t)dt, \quad x>0.
\eeqlb
In particular, one can deduce that for $0<\dz<1<A<\infty$
\beqlb\label{Wcon}
E[(W_n)^{t}1_{\{W_n<\dz\}}]\ar\rar\ar \int_{0}^{\dz}u^t\omega(u)du,\quad t>-\gz;\\
\label{Wcon1}
E[(W_n)^{t}1_{\{W_n>A\}}]\ar\rar\ar \int_{A}^{\infty}u^t\omega(u)du,\quad -\infty< t\leq 1.
\eeqlb
We also recall here a result from Lemma 13 in \cite{FW08}. There exists a constant $C>0$ such that
\beqlb\label{local}
P(Z_n=k)\leq C\l(\frac{1}{k}\wedge\frac{k^{\gz-1}}{m^{\gz n}}\r),\quad k,\, n\geq1.
\eeqlb

\begin{lem}\label{lemIL} Assume $\ez_nm^n\rar\infty$.
Then as $n\rar\infty$,
\beqlb\label{Nega0}E[Z_n^{t}L(\ez_n Z_n)]\sim
m^{nt}L(\ez_n m^n)\int_0^{\infty}u^{t}\omega(u)du,\quad  -\gz< t<1;
\eeqlb
and
\beqlb\label{Nega}
d\varliminf_{u\downarrow0}u^{1-\gz}\omega(u)\ar\leq \ar \varliminf_{n\rar\infty}\frac{E[Z_n^{-\gz}L(\ez_nZ_n)]}{ \sum_{1\leq k\leq m^n}\frac{L(\ez_n k)}{k m^{\gz n}}}\cr\ar\leq\ar\varlimsup_{n\rar\infty}\frac{E[Z_n^{-\gz}L(\ez_nZ_n)]}{ \sum_{1\leq k\leq m^n}\frac{L(\ez_n k)}{k m^{\gz n}}}\leq d\varlimsup_{u\downarrow0}u^{1-\gz}\omega(u).
\eeqlb
\end{lem}
\begin{proof} We first prove (\ref{Nega0}). Recall $W_n=Z_n/m^n.$ Note that
\beqlb\label{ILa}E\l[Z_n^{t}L(\ez_nZ_n)\r]
=m^{nt}L(\ez_nm^n)E\l[(W_n)^{t}\frac{L(\ez_nm^nW_n)}{L(\ez_nm^n)}\r].\eeqlb
 Then for $0<\dz<1<A$,
by  (\ref{regular}) and (\ref{Wcon}),  we have for some $0<\eta<\gz$ small enough, \beqlb\label{IL1}
E\l[(W_n)^{t}\frac{L(\ez_nm^nW_n)}{L(\ez_nm^n)}1_{\{W_n<\dz\}}\r]\leq CE[(W_n)^{t-\eta}1_{\{W_n<\dz\}}]=(1+o(1))C\int_0^{\dz}u^{t-\eta}\omega(u)du.
\eeqlb
Meanwhile by Dominated Convergence Theorem, we have
\beqlb\label{IL2}E\l[(W_n)^{t}\frac{L(\ez_nm^nW_n)}{L(\ez_nm^n)}\cdot1_{\{\dz \leq W_n\leq A\}}\r]\rar \int_{\dz}^Au^t\omega(u)du.\eeqlb
Finally,  using (\ref{regular}) with $\eta=1-t$, we have 
\beqlb\label{IL}
E\l[(W_n)^{t}\frac{L(\ez_nm^nW_n)}{L(\ez_nm^n)}1_{\{W_n>A\}}\r]\leq CE[W_n1_{\{W_n>A\}}]=(1+o(1))C\int_A^{\infty}u\omega(u)du.
\eeqlb
Letting $\dz\rar0$ and $A\rar\infty$, together with (\ref{ILa}),  we obtain (\ref{Nega0}). 

 The sequel of this proof is devoted to  (\ref{Nega}).  Let $\{k_n\}$ be a sequence such that $k_n\rar\infty$ and $k_n=o(m^n).$
Then for any $0<\dz\leq 1$,
$$
E[Z_n^{-\gz}L(\ez_nZ_n)]=\l(\sum_{k<k_n}+\sum_{k_n\leq k\leq\dz m^n}+\sum_{k> \dz m^n}\r)\frac{L(\ez_n k)}{k^{\gz}}P(Z_n=k)=:I_0+I_1+I_2.
$$
By Corollary 5 in \cite{FW07},  we have
$$I_1=(1+o(1))d\sum_{k_n\leq k\leq\dz m^n}\frac{L(\ez_n k)}{k^{\gz}} {m^{-n}}\omega\l(\frac{k}{m^n}\r)$$
which is larger than
\beqnn
(1+o(1))d\inf_{u\leq \dz}u^{1-\gz}\omega(u) \sum_{k_n\leq k\leq\dz m^n}\frac{L(\ez_n k)}{km^{\gz n}}
\eeqnn
and less than
\beqnn
(1+o(1))d\sup_{u\leq \dz}u^{1-\gz}\omega(u) \sum_{k_n\leq k\leq\dz m^n} \frac{L(\ez_n k)}{km^{\gz n}}.
\eeqnn
On the other hand, Dominated Convergence Theorem, together with (\ref{regular}), tells us $$I_2\sim m^{-\gz n}L(\ez_nm^n)\int_{\dz}^{\infty}u^{-\gz}\omega(u)du.$$
And we have
$$
\frac{Z_n^{-\gz}L(\ez_nZ_n)}{m^{-\gz n}L(\ez_nm^n)}1_{\{Z_n\leq \dz m^n\}}\overset{a.s.}{\rar} W^{-\gz}1_{\{W\leq \dz\}}
$$
whose expectation is infinite by (\ref{omega}). Then Fatou's lemma yields \beqlb\label{I2}\limsup_{n\rar\infty}I_2/(I_0+I_1)=0.\eeqlb By (\ref{local}), we also have
\beqlb\label{I0}
I_0\leq \sum_{k<k_n}\frac{L(\ez_n k)}{km^{\gz n}}.
\eeqlb
Then one may choose $k_n$ such that
\beqlb\label{I01}\frac{\sum_{k<k_n}\frac{L(\ez_n k)}{k}}{ \sum_{k< m^n}\frac{L(\ez_n k)}{k}}\rar0.\eeqlb
Meanwhile, one can also deduce that
\beqnn
(1+o(1))d\inf_{u\leq \dz}u^{1-\gz}\omega(u) \sum_{\dz m^n\leq k\leq m^n}\frac{L(\ez_n k)}{km^{\gz n}}
\ar\leq \ar E[{Z_n^{-\gz}L(\ez_nZ_n)}1_{\{\dz m^n\leq Z_n\leq  m^n\}}]\cr\ar\sim\ar m^{-\gz n}L(\ez_nm^n)\int_{\dz}^{1}u^{-\gz}\omega(u)du,
\eeqnn
which, together with (\ref{I0}), (\ref{I01}) and (\ref{I2}), gives
 $ \limsup_{n\rar\infty}I_0/I_1=\limsup_{n\rar\infty}I_2/I_1=0.$
Thus
\beqnn
d\inf_{u<\dz}u^{1-\gz}\omega(u)\ar\leq \ar \varliminf_{n\rar\infty}\frac{E[Z_n^{-\gz}L(\ez_nZ_n)]}{ \sum_{1\leq k\leq m^n}\frac{L(\ez_n k)}{k m^{\gz n}}}\cr\ar\leq\ar\varlimsup_{n\rar\infty}\frac{E[Z_n^{-\gz}L(\ez_nZ_n)]}{ \sum_{1\leq k\leq m^n}\frac{L(\ez_n k)}{k m^{\gz n}}}\leq d\sup_{u<\dz}u^{1-\gz}\omega(u)
\eeqnn
holds for any $\dz>0$. Letting $\dz\rar0$ implies (\ref{Nega}). We have completed the proof.
\end{proof}
\begin{remark}\label{remNN}
Lemma \ref{lemIL} could be compared with Theorem 1 in \cite{NV03} where $L=1$. Under the assumption $E[Z_1\ln Z_1]<\infty$, when $-\gamma<t<0$, our result completes the one in \cite{NV03}. However, when $t=-\gamma$, a precise limit is obtained in \cite{NV03}.
\end{remark}


\section{Proofs}\label{SecPro}

We  only prove Theorems \ref{main2}, \ref{main3}, \ref{SCmain3} and \ref{mainx}. The ideas to prove Theorems \ref{Cmain2} and \ref{Cmain3} are similar to Theorems \ref{main2} and  \ref{main3}, respectively. We omit details here.
\subsection{Proof of Theorem \ref{main2}}

\begin{lemma}\label{na5} Assume that $0<\az<1$. If $\gamma>\beta-1$,  $\ez_nm^nb(m^n)^{-1}\rar+\infty$ and $\ez_n\rar+\infty$, then there exits $\eta>0$ small enough such that for any $0<\dz<1<A,$
\beqlb\label{na51}
\limsup_{n\rar\infty}\frac{\ez_n^{\beta}(m^n)^{(\beta-1)}}{L(\ez_nm^n)}\sum_{k\leq \dz m^n} P(Z_n=k)P(S_k\geq k\ez_n)\ar\leq\ar C\dz^{\gz-\bz+1-\eta};\\
\label{na52}
\limsup_{n\rar\infty}\frac{\ez_n^{\beta}(m^n)^{(\beta-1)}}{L(\ez_nm^n)}\sum_{k\geq A m^n} P(Z_n=k)P(S_k\geq k\ez_n)
\ar\leq\ar C\int_A^{\infty}u\omega(u)du.
\eeqlb
\end{lemma}
\begin{proof} We first prove (\ref{na51}).  Consider the case of $\bz<1$.  Applying (\ref{regular}) with $0<\eta<\gz-\bz+1$, together with (\ref{Na1a}) and  (\ref{local}), gives
\beqlb
\ar\ar\sum_{k\leq \dz m^n} P(Z_n=k)P(S_k\geq k\ez_n)
\cr\ar \ar\quad\leq C\sum_{k\leq \dz m^n}P(Z_n=k)\l(kP(X_1\geq
                    r^{-1}\ez_nk)+k^{(1-\bz)r}\ez_n^{-\bz r }\r)
\cr\ar\ar\quad\leq C\l(L(\ez_nm^n)\ez_n^{-\bz }(m^n)^{1-\bz}\dz^{\gz-\bz+1-\eta}
+\dz^{(1-\bz)r+\gz}(m^n)^{(1-\bz)r}\ez_n^{-\bz r}\r).
\eeqlb
Choosing $r>1$ and noting $\ez_nm^nb(m^n)^{-1}\rar+\infty$, one can check that
\beqlb\label{o1}(m^n)^{(1-\bz)r}\ez_n^{-\bz r} \frac{L(\ez_nm^n)}{\ez_n^{\beta}(m^n)^{(\beta-1)}}=o(1).\eeqlb
Then (\ref{na51}) follows readily if $\bz<1$. The case of  $\bz\geq1$ can be proved similarly by applying (\ref{Na1a}) again with $
 r=\frac{\az\bz}{1-\az}+\bz+1
 $ and $(1-t)r=1$.

Similar reasonings also yields (\ref{na52}) by applying (\ref{regular}) with $\eta=\beta$. 
In fact, if $\bz<1$,   (\ref{Na1a}) and (\ref{local})   imply
\beqlb
\ar\ar\sum_{k\geq A m^n} P(Z_n=k)P(S_k\geq k\ez_n)\cr
\ar\ar\quad\leq  C(1+o(1))L(\ez_nm^n)\ez_n^{-\bz }m^{n(1-\bz)}\int_A^{\infty}u\omega(u)du+CA^{(1-\bz)r}(m^n)^{(1-\bz)r}\ez_n^{-\bz r},
\eeqlb
which, together with (\ref{o1}), proves (\ref{na52}) in the case of $\bz<1$.
Applying (\ref{regular}), (\ref{Na1a}) and  (\ref{local}) suitably also proves the case of  $\bz\geq1$. We omit the details here.
\end{proof}

\begin{lem}\label{na5b}
 Assume that $\gamma>\beta-1$, $\ez_nm^nb(m^n)^{-1}\rar +\infty$ and $\ez_n\rar+\infty$.  Then there exits $\eta>0$ small enough such that for any $0<\dz<1<A,$
\beqlb\label{Ia}
&&\limsup_{n\rar\infty}\bigg{|}m^{(\bz-1)n}\ez_n^{\beta}L(\ez_nm^n)^{-1}\sum_{ k=\dz m^n}^{Am^n} P( S_k\geq \ez_n k )P(Z_n=k)-I_{\bz}\bigg{|}\cr&&\qquad\leq C\l(\int_A^{\infty}u\omega(u)du+\dz^{\gamma-\bz+1-\eta}\r).
\eeqlb
\end{lem}
\begin{proof}
 Using Theorem 9.3 in \cite{DDS08} for $\az<\bz$ and using Theorem 3.3 in \cite{CH98} for $\az= \bz$,  we have  that
$$
\lim_{n\rar\infty}\sup_{x\geq x_n }\bigg{|}\frac{P(S_n\geq x)}{nP(X_1\geq x)}-1\bigg{|}=0.
$$
holds for any $x_n$ satisfying  $nF(-x_n)=o(1)$ if $\az<\bz$ or $n(1-F(x_n))=o(1)$ if $\az= \bz$.
Since $\ez_nm^nb(m^n)^{-1}\rar \infty,$ we have 
$m^nF(-\ez_n m^n)=o(1)$ if $\az<\bz$ and $m^n(1-F(\ez_nm^n))=o(1)$ if $\az=\bz$.
 In fact, if $\az<\bz$, we could denote by $b^{-1}$ the inverse of $b$. Then $\ez_nm^nb(m^n)^{-1}\rar \infty$ implies $\frac{m^n}{b^{-1}(\ez_nm^n)}\rar0$ and hence by (\ref{LB})  we have
 \beqnn
 m^nF(-\ez_n m^n)\ar=\ar \frac{m^n}{b^{-1}(\ez_nm^n)} b^{-1}(\ez_nm^n) F(-\ez_n m^n)\rar0.
 \eeqnn
If $\az=\bz$, the argument is similar. Define $$\eta_n := \sup_{\dz m^n< k<A m^n} \sup_{x\geq \ez_n k }\bigg{|}\frac{P(S_k\geq x)}{kP(X_1\geq x)}-1\bigg{|}.$$
Then one can check that $\eta_n=o(1)$ as $n\rar\infty$. Thus as $n\rar\infty$,
\beqlb\label{Ib}
\sum_{ k=\dz m^n}^{Am^n} P(Z_n=k)P( S_k\geq \ez_n k )\ar=\ar(1+o(1))\sum_{ k=\dz m^n}^{Am^n} kP(Z_n=k)P(X_1\geq\ez_n k )\cr
\ar=\ar (1+o(1))\ez_n^{-\beta}\sum_{ k=\dz m^n}^{Am^n}L(\ez_nk) k^{1-\beta}P(Z_n=k)\cr
 \ar=\ar (1+o(1))\ez_n^{-\beta}\sum_{ k=\dz m^n}^{Am^n}L(\ez_nk) k^{1-\beta}P(Z_n=k).
\eeqlb
Meanwhile,  applying  (\ref{regular}) with some $0<\eta<\gz-\bz+1$ and  (\ref{local})  yields
\beqlb\label{Ic}
L(\ez_n m^n)^{-1}m^{(\bz-1)n}\sum_{k<\dz m^n}L(\ez_nk)k^{1-\beta}P(Z_n=k)\leq C\dz^{\gamma-\bz+1-\eta}
\eeqlb
and applying  (\ref{regular}) with $\eta=\bz$ and  (\ref{local})  gives
\beqlb\label{IcA}
L(\ez_n m^n)^{-1}m^{(\bz-1)n}\sum_{k> A m^n}L(\ez_nk)k^{1-\beta}P(Z_n=k)\leq(1+o(1)) C\int_A^{\infty}u\omega(u)du.
\eeqlb

 \noindent Thus by Lemma \ref{lemIL},   we have
\beqlb\label{Id}
\ar\ar\bigg{|} m^{(\bz-1)n}L(\ez_nm^n)^{-1}\sum_{ k=\dz m^n}^{Am^n}L(\ez_nk) k^{1-\beta}P(Z_n=k)- I_{\bz}\bigg{|}\cr\ar\ar\quad\leq (1+o(1)) C\l(\int_A^{\infty}u\omega(u)du+\dz^{\gamma-\bz+1-\eta}\r).
\eeqlb
Then by (\ref{Ib}), as $n\rar\infty$,
\beqnn
\ar\ar \bigg{|}m^{(\bz-1)n}\ez_n^{\beta}L(\ez_nm^n)^{-1}\sum_{k=\dz m^n}^{Am^n}P( S_k\geq \ez_n k )P(Z_n=k)-I_{\bz}\bigg{|}\cr
\ar\ar\quad= \bigg{|}(1+o(1))m^{(\bz-1)n} \sum_{k=\dz m^n}^{Am^n}L(\ez_nk)k^{1-\beta}P(Z_n=k)-I_{\bz}\bigg{|}\cr
\ar\ar\quad\leq (1+o(1))C\l(\int_A^{\infty}u\omega(u)du+\dz^{\gamma-\bz+1-\eta}\r).
\eeqnn
The desired result follows readily.
\end{proof}

{\bf Proof of Theorem \ref{main2}:}
Letting $\dz\rar0$ and $A\rar\infty$ in Lemmas (\ref{na5}) and (\ref{na5b}) gives the theorem. \qed
\bigskip


\subsection{Proof of Theorem \ref{main3}}
Recall that $l(x)$ is an asymptotic inverse of $x\mapsto J(x)=xb(x)^{-1}$ and $l(\ez_n^{-1})=l_n$.
If $\az<\bz$, we may write \beqlb\label{repl}l(x)=x^{\frac{\az}{\az-1}}s'(x)\eeqlb for some slowly varying function $s'$.
Note that Assumption B implies that
\beqlb\label{liminfS2}
\liminf_{x\rar+\infty}s'(x)>0.
\eeqlb

\begin{lemma}\label{Na4aa}
Assume that $1<\az<2$, $p_+=0$, $\gamma>\beta-1$, $\ez_nm^nb(m^n)^{-1}\rar +\infty$  and $\ez_n\rar 0$. Then for any $0<\dz<1<A$,
\beqlb\label{main33a1}
\sum_{1\leq k\leq \dz {l_n}} P(Z_n=k)P(S_k\geq k\ez_n)\ar\leq\ar C\dz^{\gz} l_n^{\gz}m^{-\gz n},\\
\label{Na41}
\sum_{k\leq \dz m^n}P(Z_n=k)P(S_k\geq k\ez_n)\ar\leq \ar
C\dz^{\gz+1-\bz-\eta} \ez_n^{-\bz}m^{(1-\bz)n}L( \ez_nm^n)+ C l_n^{\gz}m^{-\gz n},\\
\label{Na42}
\sum_{k\geq  A m^n}P(Z_n=k)P(S_k\geq k\ez_n)\ar\leq\ar
C
\ez_n^{-\bz}m^{(1-\bz)n}L(\ez_nm^n)+ C A^{-2\gz}l_n^{\gz}m^{-\gz n},
\eeqlb
 and for any $A$ large enough,
\beqlb\label{Na47}
\ar\ar\sum_{A {l_n}< k\leq A m^n}P(Z_n=k)P(S_k\geq k\ez_n)\cr\ar\ar\quad\leq
C(1+A^{\gz+1-\bz+\eta}) \ez_n^{-\bz}m^{(1-\bz)n}L(\ez_nm^n)+C A^{-2\gz} l_n^{\gz}m^{-\gz n}.
\eeqlb
\end{lemma}
\begin{proof} The proof will be divided into three parts.

{\it Part 1:} We shall prove (\ref{main33a1}) which can be obtained by noting (\ref{local}) and
\beqnn
\sum_{1\leq k\leq \dz {l_n}}P(Z_n=k)P(S_k>k\ez_n)\ar\leq\ar\sum_{1\leq k\leq \dz {l_n}}P(Z_n=k)\cr\ar\leq\ar \frac{C}{m^{\gz n}} \sum_{1\leq k\leq\dz {l_n}}k^{\gz-1}\cr\ar\leq\ar C\dz^{\gz} l_n^{\gz}m^{-\gz n}.
\eeqnn

\smallskip

\noindent {\it Part 2:} We shall first prove (\ref{Na41}) and (\ref{Na47}).  Recall Corollary 1.6 of \cite{Na79}:
If $A_t^+:=E[X_1^{t}1_{\{X_1\geq0\}}]<\infty$ and $y^t\geq 4kA_t^+$ for some $1\leq t\leq 2$, then for $x>y$
\beqlb\label{Na1.6}
P(S_k\geq x)\leq kP(X_1>y)+(e^2kA_t^+/xy^{t-1})^{x/2y}.
\eeqlb
Thus if $s>1$, $1\leq t<\bz$ and
$$
k> \left(4 E[X_1^{t}1_{\{X_1\geq0\}}]s^{t}\right)^{1/(t-1)}\ez_n^{t/(1-t)},
$$
then
\beqlb\label{Na46}
P(S_k\geq k\ez_n)\leq kP(X_1\geq s^{-1}k\ez_n )+C(\ez_n )^{ -t s/2}k^{(1-t)s/2}.
\eeqlb
 Furthermore, (\ref{liminfS2}) implies that there exists $A_l>0$ such that (\ref{Na46}) holds for $t=\az$  and all $k>A_{l}{l_n}$.
Thus
\beqlb\label{Na43}
\ar\ar\sum_{k\leq \dz m^n}P(Z_n=k)P(S_k\geq k\ez_n)\cr\ar\ar\quad \leq\sum_{1\leq k\leq A_{l}{l_n} }P(Z_n=k)P(S_k\geq k\ez_n)+\sum_{A_{l}{l_n}<k\leq \dz m^n} P(Z_n=k)P(S_k\geq k\ez_n)\cr
\ar\ar\quad=: I_1+I_2.
\eeqlb
Applying (\ref{local}) again gives
\beqlb\label{Na44}
I_1\leq\sum_{1\leq k\leq A_{l}{l_n} }P(Z_n=k)\ar\leq\ar \frac{c}{m^{\gz n}} \sum_{1\leq k\leq A_{l}{l_n} }k^{\gz-1}\cr\ar\leq\ar CA_l l_n^{\gz} m^{-\gz n}.
\eeqlb
Note that $l_n\ez_n\rar\infty$.
Applying (\ref{Na46}) with $t=\az$, (\ref{regular}) with $\eta<\gz-\bz+1$ and (\ref{local}), we have
\beqlb\label{Na45}
I_2 \ar\leq\ar \sum_{A_{l}{l_n}<k\leq \dz m^n}P(Z_n=k)\l(kP(X_1\geq s^{-1}k\ez_n )+C(\ez_n )^{ -\az s/2}k^{(1-\az)s/2}\r)\cr
\ar\leq \ar \frac{C}{m^{\gz n}} \l(\sum_{A_l l_n< k\leq \dz m^n}{k^{\gz}}P(X_1\geq s^{-1}k\ez_n )+\sum_{k>A_{l}{l_n}}(\ez_n )^{ -\az s/2}k^{(1-\az)s/2+\gz-1}\r)
\cr\ar\leq \ar \frac{C}{m^{\gz n}} \l(L( \ez_nm^n)\sum_{ k\leq \dz m^n }\ez_n^{-\bz}k^{\gz-\bz}(k/m^n)^{-\eta}+\sum_{k>A_{l}{l_n}}(\ez_n )^{ -\az s/2}k^{(1-\az)s/2+\gz-1}\r)
\cr\ar\leq\ar
C\dz^{\gz+1-\bz-\eta} \ez_n^{-\bz}m^{(1-\bz)n}L( \ez_nm^n)+ C A_l^{-2\gz}l_n^{\gz}m^{-\gz n}s'(\ez_n^{-1})^{-2\gz},
\eeqlb
where in the last inequality, we use (\ref{repl}), (\ref{liminfS2}) and choose $s=\frac{4\gz}{\az-1}$ which implies
$ (1-\az)s/2+\gz=-\gz.$ Plugging (\ref{Na44}) and (\ref{Na45}) into (\ref{Na43}), together with (\ref{liminfS2}), gives (\ref{Na41}). Replacing $ A_{l}$ and $\dz$ by $A$ and modifying the last two steps in (\ref{Na45}) accordingly, we immediately obtain (\ref{Na47}).

\smallskip

\noindent{\it Part 3:} We shall prove (\ref{Na42}). Note that $\ez_nm^nb(m^n)^{-1}\rar +\infty$  implies $l_n\leq m^n$. Using (\ref{Na46}) with $s=\frac{4\gz}{\az-1}$ and (\ref{regular}) with $\eta=\bz$, we have
\beqnn
 \ar\ar\sum_{k\geq  A m^n}P(Z_n=k)P(S_k\ge k\ez_n)
 \cr\ar\ar\quad
 \leq \sum_{k\geq  A m^n}P(Z_n=k)\l(kP(X_1\geq s^{-1}k\ez_n )+C(\ez_n )^{ -\az s/2}k^{(1-\az)s/2}\r)
\cr\ar\ar\quad
\leq C\l(\sum_{k\geq  A m^n}P(Z_n=k)\ez_n^{-\bz}k^{1-\bz}L(s^{-1}k\ez_n )+\sum_{k>A {l_n}}(\ez_n )^{ -\az s/2}k^{(1-\az)s/2+\gz-1}m^{-\gz n}\r)
\cr\ar\ar\quad
\leq C\ez_n^{-\bz}m^{(1-\bz)n}L( \ez_nm^n)+ C A^{-2\gz}l_n^{\gz}m^{-\gz n}s'(\ez_n^{-1})^{-2\gz},
\eeqnn
where the second term in the last inequality is deduced according to similar reasonings for  (\ref{Na45}).
Then  (\ref{Na42}) follows readily.
\end{proof}

\end{remark}

\begin{lemma}\label{Na4}
Assume that $1\leq \az<2$, $p_+>0$, $\gamma>\beta-1$ , $\ez_nm^nb(m^n)^{-1}\rar +\infty$  and $\ez_n\rar \ez\in[0,\infty)$. Then there exists $\eta>0$ small  enough such that for any $0<\dz<1$,
\beqlb\label{main33a}
\sum_{k\leq  \dz m^n}P(Z_n=k)P(S_k\geq k\ez_n)\leq
C\dz^{\gz-\bz+1-\eta}L( \ez_n m^n)\ez_n^{-\bz  }m^{n(1-\bz)}.
\eeqlb
\end{lemma}
\begin{proof} The proof will be divided into three steps.

{\it Step 1:} Note that $p_+>0$ implies $\az=\bz$. We first prove that
\beqlb\label{Nabz}
P(S_k\geq \ez_n k)\leq C\left(k P(X_1\geq r^{-1}\ez_nk)+ \ez_n^{-\bz  }k^{(1-\bz)}L(\ez_n k)\right), \quad k\geq1.
\eeqlb
Recall (\ref{mu12}). By Lemma in \cite{Pr81}, we have for $k\geq1$ and $x>0$,
\beqlb\label{NaA}
P(S_k\geq x)\leq Ck\left(P(|X_1|\geq x)+\frac{\mu(2;x)}{x^2}+\frac{|\mu(1;x)|}{x}\right).
\eeqlb
%
(\ref{muasy}) implies, for $1<\bz<2$, \beqlb\label{2ndm}
\mu(2;x)=E[|X_1|^2\cdot1_{\{|X_1|\leq x\}}]\leq cx^{2-\bz}L(x),\quad x>0.
\eeqlb
On the other hand, according to (5.17), (5.21) and (5.22) in Chapter XVII in \cite{[F71]} as $x\rar\infty$,
\beqlb\label{Feller3}
\frac{x}{\mu(2;x)}E[|X_1|\cdot1_{\{|X_1|>x\}}]\rar c\neq0
\eeqlb which, together with   $E[X_1]=0$,  yields for $1<\bz<2$,
$$
|\mu(1;x)|=|E\l[X_1\cdot1_{\{|X_1|\leq x\}}\r]|\leq E\l[|X_1|\cdot1_{\{|X_1|> x\}}\r]\sim
cx^{-\bz}L(x).
$$
Thus for $1<\bz<2$, $$
|\mu(1;x)|\leq cx^{-\bz}L(x),\quad x>0.
$$ Then according to (\ref{NaA}), we obtain that (\ref{Nabz}) holds for $1<\bz<2$.


{\it Step 2:} We shall prove (\ref{Nabz}) for $\az=\bz=1$. By Theorem 1.2 in \cite{Na79}, we have
\beqlb\label{Na2}
P(S_k\geq x)\leq kP(X_1>x)+P(x),
\eeqlb
where
\beqnn P(x)=\exp\left\{1-\left(1+\frac{k\mu(2;x)-kx\mu(x)}{x^2}\right)\cdot
\log\left(\frac{x^2}{k\mu(2; x)}+1\right)\right\}.
\eeqnn
By Assumption B,
\beqnn P( x)= \frac{e k\mu(2; x)}{x^2+k\mu(2; x)}\leq \frac{e k\mu(2; x)}{x^2},
\eeqnn
which, together with (\ref{Na2}) and (\ref{2ndm}), gives that (\ref{Nabz}) holds.

{\it Step 3:} We shall prove (\ref{main33a}). By using (\ref{Nabz}), (\ref{regular}) and (\ref{local}) accordingly,
\beqnn
\ar\ar\sum_{k\leq \dz m^n}P(Z_n=k)P(S_k\geq\ez_n k)\cr
\ar\ar\qquad\leq C\left(k P(X_1\geq \ez_nk)+\ez_n^{-\bz}k^{1-\bz}L(\ez_n k)\right)\cr
\ar\ar\qquad\leq C\ez_n^{-\bz  }\sum_{k\leq \dz m^n}P(Z_n=k)k^{1-\bz}L(\ez_n k)\cr
\ar\ar\qquad\leq CL( \ez_n m^n)\ez_n^{-\bz  }\sum_{k\leq \dz m^n}P(Z_n=k)k^{1-\bz}(k/m^n)^{-\eta}\cr
\ar\ar\qquad\leq CL( \ez_n m^n)\ez_n^{-\bz  }m^{(-\gz+\eta) n}\sum_{k\leq \dz m^n}k^{\gz-\bz-\eta}\cr
\ar\ar\qquad\leq C\dz^{\gz-\bz+1-\eta}L( \ez_n m^n)\ez_n^{-\bz  }m^{n(1-\bz)}.
\eeqnn
We have completed the proof.
\end{proof}
\begin{lem}\label{main32b}
Suppose that $1<\az<2$,  $\gamma>\beta-1$,  $\ez_nm^nb(m^n)^{-1}\rar +\infty$ and $\ez_n\rar \ez\in[0,\infty).$ If  $\bz>\az$, we further assume that  \beqlb\label{XX}\lim_{n\rar\infty}\chi_n= y\in[0,\infty).\eeqlb    Then there exists $\eta>0$ small  enough such that for any $0<\dz<1$,
\beqlb\label{Ia2}
\limsup_{n\rar\infty}\bigg{|}\frac{m^{(\bz-1)n}\ez_n^{\beta}}{L(\ez_nm^n)}\sum_{ k>\dz m^n} P( S_k\geq \ez_n k )P(Z_n=k)-I_{\bz}\bigg{|}\leq C\dz^{\gamma-\bz+1-\eta}.
\eeqlb
\end{lem}
\begin{proof} First, if $1<\az<2$ and $p_+=0$, then by Theorem 9.2 in \cite{DDS08}, \beqlb\label{ldpsum}
\lim_{k\rar\infty}\sup_{x\geq x_k }\bigg{|}\frac{P(S_k\geq x)}{kP(X_1\geq x)}-1\bigg{|}=0
\eeqlb holds for any $x_k=t (\frac{\beta-\az}{\az-1}\log k)^{\frac{\az-1}{\az}}b(k), t>0.$  Define $$\eta_n := \sup_{k>\dz m^n} \sup_{x\geq \ez_n k }\bigg{|}\frac{P(S_k\geq x)}{kP(X_1\geq x)}-1\bigg{|}.$$
Then one can apply (\ref{ldpsum}) with $x_k=k\ez_n$ to ensure $\eta_n=o(1)$. To apply (\ref{ldpsum}) it suffices to show
\beqlb\label{XX1} \liminf_{n\rar\infty}\frac{m^n\ez_n}{b(\dz m^n)(\ln (m^n))^{\frac{\az-1}{\az}}}\rar+\infty.\eeqlb
In fact, since $L$ and $s$ are  slowly varying functions, then for any $\eta, \eta'>0$, there exists $C_{\eta}, C_{\eta'}$ such that
 $$L(l_n^{-1} b(l_n)m^n)\leq C_{\eta}l_n^{-\eta}b(l_n)^{\eta}m^{\eta n}$$
 and
\beqlb\label{XX2}
 \frac{l_n^{\gz-\bz}m^{(\bz-1-\gz)n} b(l_n)^{\bz}}{L(l_n^{-1} b(l_n)m^n)} \ar\geq\ar C_{\eta}\frac{l_n^{\gz-\bz+\frac{\bz}{\az}+\eta-\frac{\eta}{\az}}}{ m^{(\gz-\bz+1+\eta)n}}\frac{s(l_n)^{\bz}}{s(l_n)^{\eta}}
 \cr\ar\geq\ar
  C_{\eta}C_{\eta'}
  \frac{l_n^{\gz-\bz+\frac{\bz}{\az}+\eta-\frac{\eta}{\az}-\bz\eta'-\eta\eta'}}{ m^{(\gz-\bz+1+\eta)n}}.
 \eeqlb
Since $\az<\bz$, then one could choose $\eta, \eta'$ small enough such that
\beqlb\label{XX4}
0< \chi:=  \frac{ \gz-\bz+1+\eta}{{\gz-\bz+\frac{\bz}{\az}+\eta-\frac{\eta}{\az}-\bz\eta'-\eta\eta'}}<1.
 \eeqlb
 Thus (\ref{XX}) and (\ref{XX2}) imply
\beqlb\label{XX3}\limsup_{n\rar\infty}\frac{m^{\chi n}}{l_n}\in (0, +\infty].\eeqlb
We also note that for any $\eta''>0$,
\beqnn
\dz m^n\ez_n/b(\dz m^n)=\l(\frac{\dz m^n}{l^n}\r)^{\frac{\az-1}{\az}}\frac{s(l_n)}{s(\dz m^n)}
\geq C_{\eta''}\l(\frac{\dz m^{\chi n}}{l_n}\r)^{\frac{\az-1}{\az} } (\dz m^{(1-\chi)n})^{\frac{\az-1}{\az} }  \l(\frac{l_n}{\dz m^n}\r)^{\eta''}.
\eeqnn
Choosing $\eta''$ small enough in above, together with (\ref{XX3}) and (\ref{XX4}), yields that (\ref{XX1}) holds. We  get that $\eta_n=o(1)$. The rest proof for the case of $1<\az<2$ and $\bz>\az$ is similar to Lemma \ref{na5b}. We omit it here. 

When $1<\az=\bz<2$ and $p_+=1$, (\ref{ldpsum}) holds for $x_k$ satisfying $x_k/b(k)\rar\infty$; see \cite{R95} and references therein. Obviously, in this case $\eta_n=o(1)$.

When $1\leq\az=\bz<2$ and $0<p_+<1$, (\ref{ldpsum}) holds for $x_k$ satisfying $ kP(X_1>x_k)\rar0$ and $\frac{k}{x_k}\int_{-x_k}^{x_k}xdF(x)\rar0 $; see Theorem 3.3 in \cite{CH98}. By using (\ref{regular}), (\ref{LB}) and the fact $\ez_nm^nb(m^n)^{-1}\rar\infty$, one can check that $\eta_n=o(1)$.  Then the desired result can be proved similarly.
\end{proof}
\begin{lemma}Assume that $1<\az<2$, $p_+=0$, $\gamma>\beta-1$, $\ez_nm^nb(m^n)^{-1}\rar +\infty$  and $\ez_n\rar 0$. Then
\beqlb\label{main33b}
  V_I(\dz, A)\ar\leq\ar\varliminf_{n\rar\infty}m^{\gz n}l_n^{-\gz}\sum_{\dz l_n <k< A l_n}P(Z_n=k)P(S_k\geq k\ez_n)\cr\ar\leq\ar \varlimsup_{n\rar\infty}m^{\gz n}l_n^{-\gz}\sum_{\dz l_n <k< A l_n}P(Z_n=k)P(S_k\geq k\ez_n)\leq V_S(\dz, A),
\eeqlb
  where
  \beqnn
  V_I(\dz, A)\ar=\ar\varliminf_{u\rar0}u^{1-\gz}\omega(u)\int_{\dz}^{A}u^{\gz-1}P(U_{s}\geq u^{\frac{\az-1}{\az}})du,\cr
  V_S(\dz,A)\ar=\ar \varlimsup_{u\rar0}u^{1-\gz}\omega(u)\int_{\dz}^{A}u^{\gz-1}P(U_{s}\geq u^{\frac{\az-1}{\az}})du.
  \eeqnn
\end{lemma}
\begin{proof}Define \beqlb\label{H2}H_2=\{\dz l_n <k< A{l_n}: k= (\text{mod})d\}.\eeqlb
By Corollary 5 in \cite{FW07} and (\ref{local}),  we have
\beqlb\label{main33b1}
\ar\ar(1+o(1))d\inf_{u\leq A{l_n}m^{-n}}u^{1-\gz}\omega(u) \sum_{k\in H_2} \frac{k^{\gz-1}}{m^{\gz n}}P(S_k\geq \ez_n k)\cr \ar\ar\quad\leq
\sum_{k\in H_2}P(Z_n=k)P(S_k\geq \ez_n k)
\cr\ar\ar\quad= (1+o(1))d\sum_{k\in H_2} {m^{-n}}\omega\l(\frac{k}{m^n}\r) P(S_k\geq \ez_n k)
\cr\ar\ar\quad\leq (1+o(1))d\sup_{u\leq A{l_n}m^{-n}}u^{1-\gz}\omega(u) \sum_{k\in H_2} \frac{k^{\gz-1}}{m^{\gz n}}P(S_k\geq \ez_nk).
\eeqlb
Recall (\ref{bstable}).
Then  for any $\dz>0$
$$
\lim_{n\rar\infty}\sup_{k\in H_2}\l|P({S_k}\geq k\ez_n )-P(U_{s}\geq  k\ez_n/b(k))\r|=0.$$
Recall that $J(x)=xb(x)^{-1}$ and $l$ is the asymptotic inverse function of $J$. Then as $n\rar\infty$,
\beqlb\label{main33b2}
\sum_{k\in H_2}k^{\gz-1}P({S_k}\geq k\ez_n )\ar=\ar (1+o(1))\sum_{k\in H_2}k^{\gz-1}P\l(U_{s}\geq \frac{k\ez_n}{b(k)}\r)\cr
\ar=\ar (1+o(1))l_n^{\gz}\sum_{k\in H_2}\l(k l_n^{-1}\r)^{\gz-1}
P\l(U_{\az}\geq \frac{k\ez_n}{b(k)}\r)l_n^{-1}  \cr
\ar=\ar (1+o(1))d^{-1}l_n^{\gz}\int_{\dz}^{A}u^{\gz-1}P(U_{s}\geq u^{1-1/\az})du,
\eeqlb
where the last equality follows from the facts that
$$ \frac{k\ez_n}{b(k)}=\frac{k^{\frac{\az-1}{\az}}}{\ez_n^{-1}s(k)}\sim \frac{k^{\frac{\az-1}{\az}}}{J(l_n)s(k)}
=\frac{s(l_n)}{s(k)}\l(\frac{k}{l_n}\r)^{1-1/\az}
\quad\text{and}\quad
\lim_{n\rar\infty}\sup_{k\in H_2}\frac{s(l_n)}{s(k)}=1.
$$
Then letting $n\rar\infty$ in (\ref{main33b1}) and (\ref{main33b2}) implies the desired result by noting the fact ${l_n}m^{-n}\rar0$.
\bigskip
\end{proof}

{\it Proof of (i) in Theorem \ref{main3}:} If $\chi_n\rar0$, then we have
$$
l_n^{\gz}m^{-\gz n}m^{(\bz-1)n}\ez_n^{\beta}L(\ez_nm^n)^{-1}=o(1).
$$

Thus combining (\ref{Na41}) and Lemma \ref{main32b} together and letting $\dz\rar0$ yield the desired result. \qed

\bigskip

{\it Proof of (ii) in Theorem \ref{main3}:}
Recall H2 from (\ref{H2}).
By taking $A$ large enough in (\ref{Na42}) and (\ref{Na47}),  we have
\beqlb
\sum_{k\notin H_2}P(Z_n=k)P(S_k\geq k\ez_n)\ar=\ar\l(\sum_{1\leq k\leq \dz l_n }+\sum_{A l_n< k< A m^n}+ \sum_{k\geq A m^n}\r)P(Z_n=k)P(S_k\geq k\ez_n)
\cr\ar\leq\ar C(2+A^{\gz+1-\bz+\eta}) \ez_n^{-\bz}m^{(1-\bz)n}L(\ez_nm^n)\cr\ar\ar\qquad+C (A^{-2\gz} +\dz^{\gz}) l_n^{\gz}m^{-\gz n}.
\eeqlb
Since ${\chi_n}\rar\infty$, we have $\ez_n^{-\bz}m^{(1-\bz)n}L(\ez_nm^n)=o (l_n^{\gz}m^{-\gz n})$. Thus
$$
\varlimsup_{n\rar\infty}l_n^{-\gz}m^{\gz n}\sum_{k\notin H_2}P(Z_n=k)P(S_{k}\geq{k}\ez_n)\leq C (A^{-2\gz} +\dz^{\gz}).
$$
By (\ref{main33b}), we further have
\beqlb
V_I(\dz, A)\ar\leq\ar\varliminf_{n\rar\infty}l_n^{-\gz}m^{\gz n}P(S_{Z_n}\geq{Z_n}\ez_n)\cr\ar\leq\ar \varlimsup_{n\rar\infty}l_n^{-\gz}m^{\gz n}P(S_{Z_n}\geq{Z_n}\ez_n)\leq C (A^{-2\gz} +\dz^{\gz})+ V_S(\dz, A).
\eeqlb
Letting $\dz\rar0$ and $A\rar\infty$, together with the fact $V_I(\dz, A)\rar V_I$ and $V_S(\dz, A)\rar V_S$, yields  (\ref{main32c}). \qed

\bigskip

{\it Proof of (iii) in Theorem \ref{main3}:} Note that  ${\chi_n}\rar y\in (0,\infty)$ implies that $$l_n^{\gz}m^{-\gz n}\sim y\ez_n^{\bz}m^{(\bz-1)n}L(\ez_nm^n)^{-1}. $$
Then the desired result follows from (\ref{main33a}), (\ref{main33b}), (\ref{Na47}) and (\ref{Ia2}). \qed

\bigskip

{\it Proof of (iv) in Theorem \ref{main3}:}  Combining Lemmas \ref{Na4} and \ref{main32b} together and letting $\dz\rar0$ yield the desired result.
 We have completed the proof of Theorem \ref{main3}.\qed

\bigskip

\subsection{Proof of Theorem \ref{mainx}}
First, note that $\int_0^{\infty}P\l(U_s\geq u^{\frac{\az-1}{\az}}x\r) \omega(u)du<\infty.$ Then by (\ref{bstable}), for any $\dz>0$,
$$
\lim_{n\rar\infty}\sup_{k\geq \dz m^n}|P(S_k\geq \ez_n k)-P(U_s\geq \ez_nk/b(k))|=0.
$$
Thus
\beqnn
\sum_{k\geq \dz m^n}P(Z_n=k)P(S_k\geq\ez_nk)=(1+o(1))\sum_{k\geq \dz m^n}P(Z_n=k)P(U_{s}\geq \ez_nk/b(k)).
\eeqnn
Denote by $\bar{F}_s(x)= P(U_s\geq x)$.  Then we have
\beqnn
\ar\ar\sum_{k\geq \dz m^n}P(Z_n=k)P(S_k\geq\ez_nk)
\cr\ar\ar\quad=(1+o(1))\sum_{k\geq \dz m^n}P(Z_n=k)\bar{F}_s\l( \ez_nm^nb(m^n)^{-1}\l(\frac{k}{m^n}\r)^{\frac{\az-1}{\az}}\frac{s(m^n)}{s(k)} \r)\cr\ar\ar\quad=(1+o(1))E\l[\bar{F}_s\l( \ez_nm^nb(m^n)^{-1}\l(W_n\r)^{\frac{\az-1}{\az}}\frac{s(m^n)}{s(W_nm^n)} \r)1_{\{W_n\geq \dz\}}\r]\cr
\ar\ar\quad\rar \int_\dz^{\infty}\bar{F}_s\l( u^{\frac{\az-1}{\az}}x\r) \omega(u)du.\eeqnn
On the other hand, by (\ref{global}), as $n\rar\infty$,
\beqlb
\sum_{k\leq \dz m^n}P(Z_n=k)P(S_k\geq\ez_nk)\leq
\sum_{k\leq \dz m^n}P(Z_n=k)=(1+o(1))\int_0^\dz\omega(u)du.
\eeqlb
Letting $\dz$ go to $0$ yields the desired result.

\subsection{Proofs of Theorem \ref{SCmain3} and Corollary \ref{SCmain3Co}}
We first prove Theorem \ref{SCmain3}. Applying (\ref{Na46}) with  $\ez_n=\ez$, $k>C_s\ez^{\frac{t}{1-t}}=:A(s, t, \ez)$ and
$ s=\frac{2\gz+2}{t-1}>1$ implies
\beqnn
\ar\ar m^{\gz n}\sum_{k\geq1}P(Z_n=k)P(S_n\geq\ez k)\cr
\ar\ar\quad\leq C\sum_{k\geq1}k^{\gz-1}P(S_n\geq\ez k)\cr
\ar\ar\quad\leq C\sum_{k\leq A(s, t,\ez)}k^{\gz-1}+C\sum_{k>A(s, t,\ez)}k^{\gz-1}P(S_n\geq\ez k)\cr
\ar\ar\quad \leq CA(s, t, \ez)^{\gz}+C\sum_{k>A(s, t, \ez)}\l(k^{\gz}P(X_1\geq s^{-1}\ez k)
+\ez^{-t s/2}k^{(1-t)s/2+\gz-1}\r)\cr
\ar\ar\quad \leq CA(s, t, \ez)^{\gz}+C\sum_{k\geq1}k^{\gz}P(X_1\geq s^{-1}\ez k)
+\sum_{k>A(s, t, \ez) }\ez^{-t s/2}k^{-2}\cr
\ar\ar\quad<+\infty,
\eeqnn
where the last equality follows from $\gz<\bz-1$ or the fact that $\sum_{k\geq1}k^{\gz}P(X_1\geq s^{-1}\ez k)$ is finite because of $E[X_1^{1+\gz}1_{\{X_1>0\}}]<\infty$.
Then by dominated convergence theorem, we have
\beqlb\label{SCA}m^{\gz n}\sum_{k\geq1}P(Z_n=k)P(S_k\geq\ez k)\rar\sum_{k\geq1}q_kP(S_k\geq\ez k),\eeqlb
which yields Theorem \ref{SCmain3}. To prove Corollary \ref{SCmain3Co}, note that  by the same argument above, (\ref{SCA}) also holds with $X_1=m-Z_1$ .
Then  Corollary \ref{SCmain3Co} follows readily by applying (\ref{SCA}) twice.   \qed


\bigskip\bigskip
{\bf Acknowledgement}: We would like to give our sincere thanks to Dr. Weijuan Chu and Professors Xia Chen and Chunhua Ma for their enlightening discussions.


\bigskip


\begin{thebibliography}{12}
\bibitem{A94}
K. B. Athreya:
Large deviation rates for branching processes-I. Single type case. {\em The Annals of Applied Probability} 4:779--790, 1994.

\bibitem{AN72}
K. B. Athreya and P. E. Ney: {\em Branching Processes}, Springer, Berlin, 1972.

\bibitem{AV97}
K. B. Athreya, A. N. Vidyashankar: Large deviation rates for supercritical and critical branching
processes. In: {\em Classical and Modern Branching Processes, IMA Volumes in Mathematics and its Applications,} 84:1--18, Springer, Berlin, 1997.


\bibitem{BGT89}
N. H. Bingham, C. M. Goldie, J. L. Teugels: {\em Regular Variation}, Cambridge University Press, 1989.

\bibitem{CLR14}
W. Chu, W. V. Li and Y.-X. Ren:
Small value probabilities for supercritical branching processes with Immigration
{\em Bernoulli} 20: 377--393, 2014.

\bibitem{CH98}
D. B. H. Cline and T. Hsing:
 \newblock Large deviation probabilities for sums of random variables with heavy or
subexponential tails.
\newblock {\em Technical Report} Texas A\& M University, 1998.

\bibitem{DDS08}
D. Denisov, A. B. Dieker and V. Shneer:
\newblock Large deviations for random walks under subexponentiallity: the big-jump domain.
\newblock {\em Ann. of Probab.} 36:1946--1991, 2008.

\bibitem{D71}
 S. Dubuc:
 Probl\`emes relatifs \`a l'it\'eration de fonctions sugg\'er\'es par les processus en cascade.
 {\em Ann. Inst. Fourier (Grenoble)} 21:171--251, 1971.



\bibitem{[F71]}
{} W. Feller,: \textit{An Introduction to Probability Theory
and Its Applications} {\bf 2}, 2nd ed, Wiley, New York, 1971.

\bibitem{FW07}
K. Fleischmann and V. Wachtel:
\newblock Lower deviation probabilities for superciritcal Galton-Watson processes.
\newblock {\em Ann. Inst. H. Poincar\'e}  43:233--255, 2007.


\bibitem{FW08}
K. Fleischmann and V. Wachtel:
\newblock Large deviations for sums indexed by the generations of a Galton-Watson process.
\newblock {\em Probab. Theory and Related Fields} 141:445--470, 2008


\bibitem{JP96}
C. Jacob and J. Peccoud: Inference on the initial size of a supercritical branching processes from migrating binomial observations. {\em C. R. Acad. Sci. Paris Ser. I} 322:875--888, 1996.

\bibitem{JP98}
C. Jacob and J. Peccoud: Estimation of the parameters of a branching process from migrating binomial observations, {\em Adv. Appl. Probab.} 30:948--967, 1998.

  \bibitem{Na69}
A. V. Nagaev:
\newblock On estimating the expected number of direct descendants of a particle in a branching process.
\newblock {\em Theory Probab. Appl.} 12:314--320, 1967.

\bibitem{Na79}
S. V. Nagaev:
\newblock Large deviations of sums of independent random variables.
\newblock {\em Ann. Probab.} 7:745--789, 1979.

\bibitem{NV03}
 P. E. Ney and A. N. Vidyashankar: Harmonic moments and large deviation rates for supercritical branching
processes. {\em Ann. Appl. Probab.} 13:475--489, 2003.

\bibitem{NV04}
 P. E. Ney and A. N. Vidyashankar: Local limit theory and large deviations for superciritcal branching processes {\em Ann. Appl. Probab.} 14:1135--1166, 2004.

\bibitem{Pi04}
D. Piau: Immortal branching Markov processes: averaging properties and PCR applications. {\em Ann. Probab.} 32:337--363, 2004.

 \bibitem{Pr81}
 W. E. Pruitt: The growth of random walks and Levy
 processes. {\em Ann. Prob.} 9:948--956, 1981.

\bibitem{R95}
L. V. Rozovski\u{i}: Probabilities of large deviations of sums of independent random
variables with a common distribution function from the domain of attraction of a nonsymmetric
stable law. (Russian. Russian summary). {\em Teor. Veroyatnost. i Primenen.} 42: 496--530, 1997; translation in {\em Theory Probab. Appl.} 42:454--482, 1998.

\bibitem{ST94}
G. Samorodnitsky and M. Taqqu: {\em Stable Non-Gaussian Random Processes: Stochastic Models with Infinite Variance}, (Stochastic Modeling Series)-Chapman and Hall, New York, 1994.

\bibitem{MW13}
T.~Mikosch and O. Wintenberger:
Precise large deviations for dependent regularly varying sequences.
{\em Probab. Theory and Related Fields} 156:851--887, 2013.

\bibitem{MW14}
T. Mikosch and O. Wintenberger: The cluster index of regularly varying sequences with applications to limit theory for functions of multivariate Markov chains. {\em Probab. Th. Rel. Fields} 159:157-196, 2014.




\end{thebibliography}


\end{document}